\documentclass[11pt, oneside]{article}

\usepackage[a4paper, left=3cm,top=3cm,right=3cm,bottom=3cm]{geometry}
\usepackage{tikz}

\usepackage{hyperref}

\usepackage{psfrag}
\usepackage{faktor}

\usepackage[sfdefault]{FiraSans}

\usepackage[T1]{fontenc}

\usepackage[onehalfspacing]{setspace}

\newcommand{\skl}[1]{(#1)}

\usepackage{graphicx}
\usepackage{amssymb}
\usepackage{amsmath}

\usepackage{authblk}
\usepackage{amsthm}
\newtheorem{theorem}{Theorem}
\newtheorem{problem}[theorem]{Problem}

\newtheorem{proposition}[theorem]{Proposition}

\newtheorem{remark}[theorem]{Remark}
\newtheorem{example}[theorem]{Example}

\theoremstyle{definition}
\newtheorem{condition}[theorem]{Condition}
\newtheorem{definition}[theorem]{Definition}

\usepackage{siunitx}
\usepackage{bm}
\usepackage{mathtools}

\usepackage{enumitem}
\setitemize{label=\scriptsize{$\blacksquare$},topsep=0em}
\setenumerate{label=(\alph*), topsep=0em}

\newcommand{\risk}{R}
\newcommand{\al}{\alpha}
\newcommand{\la}{\lambda}

\newcommand{\eps}{\epsilon}

\newcommand{\nullnet}{\mathbf{N}}  
\newcommand{\nun}{{\bm \Phi}}

\newcommand{\Cone}{\mathbf C}
\newcommand{\R}{\mathbb R}

\newcommand{\N}{\mathbb N}

\newcommand{\X}{\mathbb X}
\newcommand{\Y}{\mathbb Y}

\newcommand{\Fo}{\mathbf{F}}
\newcommand{\Ao}{\mathbf{A}}

\newcommand{\Ko}{\mathbf{K}}
\newcommand{\Bo}{\mathbf{B}}
\newcommand{\Ho}{\mathbf{H}}
\newcommand{\So}{\mathbf{S}}
\newcommand{\Po}{\mathbf P}
\newcommand{\Do}{\mathbf D}

\newcommand{\Ro}{\mathbf R}

\newcommand{\M}{\mathcal M}

\newcommand{\reg}{\mathcal{R}}

\newcommand{\tik}{\mathcal{T}}
\newcommand{\tikl}{\mathcal{L}}
\newcommand{\B}{\boldsymbol{\Delta}}
 
\newcommand\sabs[1]{{\lvert#1\rvert}}

\newcommand{\Id}{\operatorname{Id}}

\def\plus{{\boldsymbol{\dag}}}

\newcommand\norm[1]{\Vert#1\Vert}
\newcommand\set[1]{{\{#1\}}}
\newcommand\kl[1]{{\left(#1\right)}}

\newcommand\snorm[1]{\Vert#1\Vert}
\newcommand{\ran}{\operatorname{ran}}

\DeclareMathOperator*{\argmin}{arg\,min}
\DeclareMathOperator{\id}{Id}

\DeclareMathOperator{\Dom}{dom}

\DeclareMathOperator{\func}{\mathbb F}

\DeclarePairedDelimiter{\innerprod}{\langle}{\rangle}
\DeclarePairedDelimiter{\abs}{\lvert}{\rvert}

\newcommand{\signal}{{\bf x}}
\newcommand{\zsignal}{{\bf z}}

\newcommand{\data}{\mathbf y}

\newcommand{\zz}{\mathbf z}

\newcommand{\noise}{\xi}

\usepackage{soul}
\colorlet{lred}{red!40}
\colorlet{lgreen}{green!40}
\colorlet{lblue}{blue!40}
\definecolor{bananamania}{rgb}{0.98, 0.91, 0.71}

\allowdisplaybreaks

\newcommand{\XX}{\mathbb X}
\newcommand{\YY}{\mathbb Y}
\newcommand{\DD}{\mathbb D}
\newcommand{\MM}{\mathbb M}
\newcommand{\Wo}{\mathcal V}

\newcommand{\simm}{\mathcal{D}}  
\newcommand{\dd}{\mathbf{d}}

\newcommand{\UU}{\mathbb U}
\newcommand{\VV}{\mathbb V}

\newcommand{\nlf}{\psi}          
\newcommand{\NN}{{\mathbf{\Phi}}}
\newcommand{\XXX}{\mathbf{\Xi}}
\newcommand{\edot}{\,\cdot\,}
\newcommand{\A}{\mathbf A}

\newcommand{\No}{\mathcal N}
\newcommand{\D}{\mathcal D}

\newcommand{\nlo}{\sigma}       
\newcommand{\breg}{\mathcal{B}}   
\newcommand{\modt}{\nu}

\newcommand{\decoder}{\mathbf E}
\newcommand{\ddecoder}{\mathbf D}

\author{Markus~Haltmeier}

\affil{Department of Mathematics.
University of Innsbruck
Technikerstrasse 13, 6020 Innsbruck, Austria.
 {\tt markus.haltmeier@uibk.ac.at}}

\author{Linh Nguyen}

\affil{ Department of Mathematics, University of Idaho,
Moscow, ID 83844, US, {\tt lnguyen@uidaho.edu}
}

\numberwithin{equation}{section}
\numberwithin{figure}{section}
\numberwithin{theorem}{section}

\title{Regularization of Inverse Problems by Neural Networks}
\date{June 6, 2020}							

\begin{document}
\maketitle

\begin{abstract}
Inverse problems arise in a variety of imaging applications including computed  tomography,  non-destructive testing, and  remote sensing.  The characteristic features  of inverse problems are  the non-uniqueness and instability of their solutions.  Therefore,  any reasonable solution method requires the use of regularization tools that  select specific solutions and at the same  time stabilize the inversion process.  Recently, data-driven methods using deep learning techniques and neural networks demonstrated to significantly outperform classical solution methods for inverse problems. In this chapter, we give an overview of  inverse problems and demonstrate the necessity of    regularization concepts for  their solution. We show that neural networks can be used for the data-driven solution of inverse problems and review  existing deep learning methods for inverse problems. In particular, we view these deep learning methods from the perspective of regularization theory, the mathematical foundation of stable solution methods for inverse problems.  This chapter is  more than just a review as many of the presented theoretical results extend existing ones.

\end{abstract}

\tableofcontents

 \setlength{\parskip}{0.8em}

\section{Introduction} \label{sec:intro}

The solution of inverse problems arises in a variety of practically important applications including medical imaging, computer vision, geophysics as well as many other branches of pure and applied sciences. Inverse problems are most efficiently formulated as an estimation  problem of the form  
\begin{equation}\label{eq:ip}
	\text{recover} \quad  \signal^* \in \XX \quad \text{from data } \quad
	\data = \A( \signal^* )   +  \noise  \in \YY    \,.
\end{equation}
Here $\Ao  \colon  \XX \to \YY$ is a mapping between normed spaces, $\signal^* \in  \XX$ is  the true unknown solution, $\data$ represents the given data, and $\noise$ is an unknown data perturbation. In this context, the application of the operator $\A$ is referred to as the forward operator or forward problem, and solving \eqref{eq:ip} is the corresponding inverse problem.  In the absence of noise where $\xi = 0$ we refer  to $\data = \A( \signal^* )$ as exact data,  and in the case where $\noise \neq 0$ we refer to  $\data$ as noisy data. 

Prime examples of inverse problems are image reconstruction problems, where the forward operator describes the data generation process depending on the image reconstruction modality. For example, in X-ray computed tomography (CT), the forward operator is the sampled Radon transform whereas in magnetic resonance imaging (MRI) the forward operator is the sampled Fourier transform. Reconstructing the diagnostic image from experimentally collected data amounts to solving an inverse problem of the form \eqref{eq:ip}. In these and other applications, the underlying forward operator is naturally formulated between infinite dimensional spaces, 
because the object to be reconstructed is a function of a continuous spatial variable.  Even though the numerical solution is performed in a finite dimensional discretization, the mathematical properties of the continuous formulation are crucial for understanding and improving image formation algorithms.

\subsection{Ill-posedness}

The inherent character of inverse problems is their ill-posedness.
This means that even in the  case of  exact data, the solution of \eqref{eq:ip} is either not unique, not existent,  or does not stably depend on the given data.  More formally, for an inverse problem  at least one of the following three unfavorable properties holds:
\begin{enumerate}[label=(I\arabic*)]
\item \label{ip1} 
\textsc{Non-uniqueness:} 
For some $\signal_1^* \neq  \signal_2^* \in \XX$ we have $\Ao (\signal^*_1) = \Ao (\signal^*_2)$.

\item \label{ip2} 
\textsc{Non-existence:} For some $\data \in \YY$, the equation $\Ao(\signal) = \data$ has no solution.

\item \label{ip3} 
\textsc{Instability:} 
Smallness of $\norm{\Ao (\signal^*_1) - \Ao (\signal^*_2)}$ does not imply smallness of $\norm{\signal^*_1  - \signal^*_2}$.      
\end{enumerate} 
These conditions imply that the forward operator does not have a continuous inverse which could be used to directly solve \eqref{eq:ip}. Instead, regularization methods have to be applied which result in stable methods for solving  inverse problem.

Regularization methods approach the ill-posedness by two steps. 
Firstly,  to address non-uniqueness  and non-existence issues \ref{ip1}, \ref{ip2}, one restricts the image and pre-image space of the forward operator to sets  $\MM \subseteq \XX $ and $\ran(\Ao) \subseteq \YY$ such that the restricted forward operator  $\Ao_{\rm res} \colon  \MM \to \ran(\Ao)$ becomes bijective. For any exact data, the equation  $\Ao (\signal)  = \data$ then has a unique solution in $\MM$  which is given by the inverse of the restricted forward operator applied to $\data$.  Secondly, in order to  address the instability issue \ref{ip3}, one considers a family of continuous operators $\Bo_ \alpha \colon \YY \to \XX$ for $\alpha > 0 $ that  converge to $\Ao_{\rm res}^{-1}$ in a suitable sense;  see Section~\ref{sec:background}  for  precise definitions.

Note that the choice of the set  $\MM$ is crucial as it represents the class of desired reconstructions and acts as selection criteria for picking a particular solution of  the given inverse problem. A main challenge is that this class  is actually unknown or at least it cannot be described properly. For example, in CT for medical imaging,  set of desired solutions   represents the set of all functions corresponding to  spatially attenuation inside patients,  a function class that is clearly challenging, if not impossible, to describe in mathematical terms.   
              
Variational regularization and variants \cite{scherzer2009variational} have been the most successful class of regularization methods for solving  inverse problems. Here, $\MM$ is  defined as solutions having a small value of a certain regularization functional that can be interpreted as a measure  for the deviation  from desired solutions. Various regularization functionals have been  analyzed for inverse problems, including  Hilbert space norms \cite{engl1996regularization}, total variation \cite{acar1994analysis} and   sparse $\ell^q$-penalties \cite{daubechies2004iterative,grasmair2008sparse}.  
Such handcrafted regularization functionals  have limited complexity and are unlikely to accurately model complex signal classes arising in applications such as medical imaging. On the other hand, their regularization effects are well understood, efficient numerical algorithms have been developed for their realization, they work reasonably well in practice, and they have been rigorously analyzed mathematically.

\subsection{Data-driven reconstruction}

Recently data-driven methods based on neural networks and deep learning demonstrated to significantly outperform existing variational and iterative reconstruction algorithms for solving inverse problems. The essential  idea is to use neural networks to define a class $(\Ro_\theta)_{\theta \in \Theta}$ of reconstruction networks $\Ro_\theta \colon  \YY \to \XX$ and to select the parameter  vector $\theta \in \Theta$   of the network in data-driven manner.  The selection is based on 
a set of training data $(\signal_1, \data_1), (\signal_2, \data_2), \dots, (\signal_N, \data_N) $ where $\signal_i \in \MM$ are desired reconstructions and  $\data_i  =  \Ao(\signal_i^*) + \noise_i \in  \YY$ are corresponding data. Even if the set  $\MM$  of desired reconstructions  is unknown, the available samples $\signal_1, \dots, \signal_N$ can be used to select the particular reconstruction method. 
A typical selection strategy is to minimize  a penalized  least squares functional having the form 
\begin{equation} \label{eq:risk-lsq}
 	 \theta^* 
	\in \argmin_\theta  
 	\left\{ \frac{1}{N} \sum_{i=1}^N  \bigl\Vert \signal_i  - \Ro_\theta (\data_i)  \bigr\|^2  +  P( \theta )  \right\}\,.
\end{equation}
The final neural network based reconstruction method is then given by 
$\Ro_{\theta^*} \colon  \YY \to \XX$ and is such that in average  it performs well on the given training data set.    

Existing deep learning-based methods include post-processing networks \cite{han2016deep,jin2017deep}, null-space networks \cite{schwab2019deep,schwab2020big}, variational networks \cite{kobler2017variational}, iterative networks \cite{ADMMnet,adler2017solving,aggarwal2018modl}, network cascades \cite{kofler2018u,schlemper2017deep} and learned regularization functional   \cite{li2018nett,Lunz2018,obmann2020deep}.   
We refer  to the review  \cite{arridge2019solving} for other  
 data-driven reconstruction methods such as  GANs \cite{bora2017compressed,mardani2018deep}, dictionary learning, deep basis pursuit \cite{sulam2019multi} or deep image priors \cite{ulyanov2018deep,van2018compressed,dittmer2019regularization}, that we do not touch in this chapter.  Post-processing networks  and null-space networks are explicit, where  the reconstruction network  is given explicitly and its  parameters are  trained to fit the given training data. Methods using  learned regularizers are implicit and  the  reconstruction  network  $\Ro_\theta (\data)=  \argmin \tik_{\theta, \data} $  is defined by minimizing a properly trained Tikhonov functional  $\tik_{\theta, \data} \colon \XX \to [0, \infty]$. Variational networks and iterative networks are in  between, where $ \argmin \tik_{\theta, \data}$ is approximated  via an iterative scheme using $L$ steps.

 Any reasonable method for solving an inverse problem, including all learned reconstruction schemes,  has to include some form of regularization. However, regularization may be imposed implicitly, even without noticing by the researcher developing the algorithm. Partially, this is the case because discretization,  early stopping, or other techniques to numerically stabilizing an optimization algorithm  at the same time has a regularization effect on  the underlying inverse problem. Needless to say, understanding  and analyzing where exactly the  regularization effect comes from,  will increase reliability of any algorithm and allows its further improvement. In conclusion, any data-driven reconstruction method has to include  either explicitly or implicitly a form of regularization. In this chapter, we will analyze the regularization properties of various deep learning methods for solving inverse problems.

\subsection{Outline}

The outline of this chapter is as follows. In Section~\ref{sec:background} we present the background of inverse problems and deep learning. In Section~\ref{sec:regnet}  we analyze direct neural network based  reconstructions whereas  in Section~\ref{sec:nett} we study variational and iterative  reconstruction methods  based on neural  networks. The chapter concludes with a discussion and some final remarks given in Section~\ref{sec:conclusion}.  While  the concepts presented in the  subsequent sections are known, most of the presented  results extend existing  ones.  Therefore, this chapter is much more than just a review over existing results.

For the sake of clarity, in  this chapter we only study linear inverse problems  even several results can be extended non-linear problems as well. We will provide remarks pointing to such results. Throughout, we allow an infinite dimensional setting, because in many applications the unknowns to be recovered as well as the data are  most naturally modeled as functions which lie in infinite dimensional spaces $\XX$ and $\YY$. However, everything said in this chapter applies to finite dimensional spaces as well.  In limited data problems, such as sparse view CT, the finite dimension of the data space $\YY$ is even an intrinsic part of the forward model.  Therefore, the reader not familiar with infinite dimensional vector space can  think of  $\XX$ and  $\YY$  as finite dimensional vector spaces  each equipped with a  standard vector norm.

\section{Preliminaries}
\label{sec:background}

In this section, we provide necessary background on linear inverse problems, their regularization and their solution with neural networks.

Throughout the following,  $\XX$ and $\YY$ are  Banach spaces.
We study solving inverse problems of the form \eqref{eq:ip} in a deterministic setting with a bounded linear forward operator $\Ao  \colon  \XX \to \YY$. Hence we aim for 
estimating   the  unknown signal $\signal^* \in \XX$ from the available data $ \data = \A( \signal^* )   +  \noise $, where   $\noise \in  \YY$ is the  noise that is assumed to satisfy an estimate of the form $\norm{\noise} \leq \delta$. Here  $\delta \geq 0$ is called  the  noise level and in the case $\delta = 0$ we call $\data = \A( \signal^* )$ the exact data.

\subsection{Right inverses}

As we have explained  in the  introduction, the main feature of 
inverse problems is their ill-posedness.  Regularization methods approach the ill-posedness by two steps. In the first step, they address \ref{ip1} and \ref{ip2} by restricting the image and the pre-image spaces which gives a certain right inverse defined on $\ran(\Ao)$. In order to address the instability  issue  \ref{ip3}, regularization methods are applied for stabilization.  We first consider right inverse and their instability, and consider the regularization in the following subsection.

\begin{definition}[Right inverse]  \label{def:right}
A possibly non-linear mapping $\Bo \colon \ran(\Ao) \subseteq \YY \to \XX$ is called  right  inverse of $\Ao$ 
if $\Ao( \Bo (\data)) = \data$ for all $\data \in \ran(\Ao)$.  
\end{definition}

Clearly, a right inverse always exists because  for any  $\data \in \ran(\Ao)$ there exists an element $\Bo \data  \coloneqq \signal$ such that  $\Ao\signal = \data$. However, in general no continuous right inverse exists. More precisely, we have the following result (compare \cite{nashed1987inner}). 

\begin{proposition}[Continuous right inverses] \label{prop:right}
Let $\Bo  \colon \ran(\Ao) \to \XX$ be a continuous right inverse.  Then  $\ran(\Ao)$ is closed.
\end{proposition}

\begin{proof}
By continuity, $\Bo$ can be extended in a unique way  to a mapping   $\Ho \colon \overline{\ran(\Ao)} \to \XX$. Let 
$\signal_n \in \ran(\Ao)$ with $\signal = \lim_{n \to \infty} \signal_n \in \overline{\ran(\Ao)}$. The continuity of  $\Ho$ and $\Ao$   implies $\Ao \circ  \Ho = \id_{\overline{\ran(\Ao)}}$.  
Therefore  $\overline{\ran(\Ao)} = \ran(\Ao \circ \Ho)  \subseteq \ran(\Ao)  \subseteq \overline{\ran(\Ao)}$ which  shows that $ \ran(\Ao)  $ is closed.
\end{proof}

Proposition \ref{prop:right} implies that whenever  $\ran(\Ao)$ is non-closed, then $\Ao$ does not have a continuous right inverse.  
 
The next question we study is the existence of a linear right inverse.
For that purpose  recall that a mapping $\Po \colon \XX \to \XX$ is called projection  if  $\Po^2 = \Po$. If $\Po$ is a linear bounded projection, then      
$\ran(\Po) $ and  $ \ker(\Po)$ are closed subspaces and $\XX = \ran(\Po) \oplus  \ker(\Po)$.      

\begin{definition}[Complemented subspace]
A closed subspace  $\VV$  of $\XX$ is called complemented  in $\XX$, if there exists a bounded linear  projection $\Po$ with $\ran(\Po) =  \VV$
\end{definition}

A closed subspace  $\VV \subseteq \XX$ is complemented
if and only if there is another  closed subspace $\UU \subseteq \XX$ with  $\XX= \UU \oplus \VV$. 
In a Hilbert space, any  closed subspace is complemented, and $\XX =  \VV^\bot \oplus \VV$ with the orthogonal complement $\VV^\bot \coloneqq  \set{u \in \XX \mid \forall v \in \VV \colon  
\innerprod{u,v}  = 0}$.  However, as shown in \cite{lindenstrauss1971complemented}, in every Banach space that is  not isomorphic to a Hilbert space there exist closed subspaces which are not complemented. 

\begin{proposition}[Linear right inverses] \label{prop:rightlin}\mbox{}
\begin{enumerate}
\item\label{prop:rightlin1} $\Ao$ has a linear right inverse if and only if $\ker(\Ao)$ is complemented.

\item\label{prop:rightlin2}  A linear right inverse is continuous if and only $\ran(\Ao)$ is closed.  
\end{enumerate}
\end{proposition}

\begin{proof}\mbox{}
\ref{prop:rightlin1}  First suppose that $\Ao$ has a linear right inverse $\Bo \colon \ran(\Ao) \to \XX$.
For any    $\signal \in \XX$ we have  $(\Bo \circ \Ao)^2 (\signal) =  \Bo \circ   ( \Ao \circ \Bo) (\Ao (\signal) ) = (\Bo \circ \Ao) (\signal)$. Hence  $\Bo \circ \Ao$ is a linear bounded projection.  This implies  the topological decomposition  
$\XX = \ran(\Bo \circ \Ao ) \oplus  \ker( \Bo \circ \Ao)$
with closed subspaces $\ran(\Bo \circ \Ao )$ and  $\ker(\Bo \circ \Ao)$.
It holds   $\ker( \Bo \circ \Ao)  \supseteq \ker(\Ao)  = \ker(\Ao \circ \Bo \circ \Ao)
\supseteq  \ran(\Bo \circ \Ao )$ which  shows that $ \ker(\Ao)  = \ran(\Bo \circ \Ao) $ is complemented. Conversely let  $\ker(\Ao) $ be complemented and   write   $\XX = \XX_1 \oplus  \ker(\Ao)$. Then $\Ao_{\rm res} \colon \XX_1 \to \ran(\Ao)$ is bijective and therefore has a linear inverse $\Ao_{\rm res}^{-1}$ defining a right inverse for $\Ao$.  

\ref{prop:rightlin2}
For any continuous right inverse, $\ran(\Ao)$ is closed according   to Proposition~\ref{prop:right}. Conversely, let  $\Bo \colon \ran(\Ao) \to \XX$.
be linear right inverse and $\ran(\Ao)$ closed. In particular,   $\ker(\Ao) $ is  complemented and we can write $\XX = \XX_1 \oplus  \ker(\Ao)$.
The restricted mapping  $\Ao_{\rm res} \colon \XX_1 \to  
\overline{\ran(\Ao)}$ is bijective therefore bounded according to the bounded inverse theorem.  This implies that $\Bo$ is bounded, too.          
\end{proof}

In a Hilbert space $\XX$ the kernel  $\ker(\Ao)$ of a bounded linear operator is complemented, as any other closed subspace of $\XX$. Therefore, according to Proposition~\ref{prop:rightlin}, any bounded linear operator defined on a Hilbert space has a linear right inverse.   However, in a general Banach space this is not the case, as the following example shows.

\begin{example}[Bounded linear operator without linear right inverse]\label{eq:noinverse}
Consider the set $c_0(\N)$ of  all sequences converging to zero as a subspace of  the space $\ell^\infty(\N)$ of all bounded sequences $\signal \colon \N \to \R$ with the supremum  norm  $\norm{\signal}_\infty \coloneqq 
\sup_{ n \in \N} \abs{\signal(n)}$. Note that $c_0(\N) \subseteq \ell^\infty(\N)$ is a classic example for a  closed subspace that is not complemented in a Banach space, as first shown in  \cite{phillips1940linear}.  Now consider  the quotient space $\YY= \faktor{ \ell^\infty(\N)}{c_0(\N)} $ where  elements in $\ell^\infty(\N)$ are identified  if their difference is contained in $c_0(\N)$. Then the quotient map $\Ao \colon \ell^\infty(\N) \to \YY \colon \signal \mapsto [\signal]$ is clearly linear, bounded, and onto with   $\ker(\Ao) = c_0(\N)$. 
It is clear that a right inverse of $\Ao$ exists which can be constructed  by simply choosing any representative in $[\signal]$. However, because $c_0(\N)$ is not complemented, the kernel of $\Ao$ is not complemented and according to Proposition~\ref{prop:rightlin} no linear right inverse  of $\Ao$ exists.
\end{example}

At first glance it might be  surprising that bounded linear forward operators  do not always have linear right inverses.  However, following Example~\ref{eq:noinverse} one constructs bounded linear operators without linear right inverses for every Banach space that is not isomorphic to a Hilbert space. This  in particular includes  the function spaces $L^p(\Omega)$ with $p \neq 2$, where inverse problems are often formulated on.

\begin{proposition}[Right inverses in Hilbert spaces] \label{prop:right-hil}
Let $\XX$ be a Hilbert space and   let $\Po_{\ker(\Ao)} \colon \X \to \X$ denote the orthogonal projection onto $\ker(\Ao)$.  
\begin{enumerate}
\item\label{prop:right-h1}  
$\Ao$ has a unique linear right inverse $\Ao^\plus \colon \ran(\Ao) \to \XX$ with 
$\Ao \circ \Ao^\plus  = \id  - \Po_{\ker(\Ao)}$. 

\item\label{prop:right-h2}  
$\forall \data \in \ran(\Ao) \colon \Ao^\plus (\data) = \argmin\set{\norm{\signal} \mid \Ao (\signal) = \data }$.

\item\label{prop:right-h3}  
$\Ao^\plus$  is continuous if and only if $\ran(\Ao)$  
is closed.  

\item\label{prop:right-h4}   
If $\ran(\Ao)$ is non-closed, then any right inverse is discontinuous.  
\end{enumerate}
\end{proposition}

\begin{proof}
In a Hilbert space the orthogonal  complement $\ker(\Ao)^\bot$
defines a complement of $\ker(\Ao)$ and therefore   
\ref{prop:right-h1}, \ref{prop:right-h3},  \ref{prop:right-h4}
follow from Propositions  \ref{prop:right}  and \ref{prop:rightlin}.
 Item \ref{prop:right-h2} holds because   any solution 
 of the equation $\Ao (\signal) = \data$ has the form $\signal = \signal_1 + \signal_2 
 \in \ker(\Ao)^\bot \oplus \ker(\Ao)$ and 
we have  $\norm{\signal}^2 = \norm{\signal_1}^2  + \norm{\signal_2}^2$
 according to Pythagoras theorem.  
\end{proof}

In the case that $\XX$ and $\YY$ are both Hilbert spaces 
 there is a unique extension  $\Ao^\plus \colon \ran(\Ao) \oplus \ran(\Ao)^\bot \to \XX$ such that   $\Ao^\plus(\data_1 \oplus \data_2 ) = \Ao^\plus(\data_1)$ for all  $\data_1 \oplus \data_2 \in  \ran(\Ao)  \oplus  \ran(\Ao)^\bot$. The operator $ \Ao^\plus$ is referred to as the Moore-Penrose inverse of $\Ao$. For more background on generalized in inverses in Hilbert and Banach spaces see \cite{nashed1987inner}.

\subsection{Regularization methods}
\label{sec:reg}

Let $\Bo \colon \ran(\Ao) \subseteq \YY \to \XX$  be a right inverse 
of $\Ao$, set $\MM \coloneqq \ran(\Bo)$ and suppose $\MM^* \subseteq \MM$. Moreover, let $\simm \colon \YY \times \YY \to [0, \infty] $ be  some functional measuring closeness in the data space. The standard choice is  the squared norm distance $\dd_{\YY}(\data, \data^\delta)  =  \norm{\data -  \data^\delta}$ but also other choices will be considered in this chapter.      

\begin{definition}[Regularization method]\label{def:reg}
A function $\Ro \colon  (0, \infty) \times  \YY  \to \XX$ with  
\begin{equation}
\forall \signal \in \MM^*
\colon  \lim_{\delta \rightarrow 0} \sup \Bigl\{\norm{\signal -  \Ro (\delta, \data^\delta)}   \mid \data^\delta \in \YY \wedge  \simm(\Ao (\signal), \data^\delta)  \leq \delta \Bigr\} = 0 \,.
\end{equation}
is called  (convergent) regularization method for \eqref{eq:ip}  on the signal  class   $\MM^* \subseteq \MM$ with respect to the  similarity measure $\simm$. We also write $(\Ro_\delta)_{\delta >0}$ instead of $\Ro$.    
\end{definition}

The following lemma gives a useful  guideline for creating regularization 
methods based on point-wise approximations of $\Bo$.     

 \begin{proposition}[Point-wise approximations are  regularizations]\label{prop:point}
Let $(\Bo_\alpha)_{\al>0}$ be a family of continuous operators $\Bo_\alpha \colon \YY\rightarrow \XX$ that converge uniformly to $\Bo$ on    
$\Ao(\MM^*) $ as $\al \to 0$. Then, there is a function $\al_0\colon (0, \infty) \to (0, \infty)$ such that  
\begin{equation}\label{eq:reg}
	\Ro  \colon  (0,\infty) \times \YY  \to \XX \colon 
	(\delta, \data^\delta) \mapsto \Ro (\delta, \data^\delta)
	\coloneqq \Bo_{\alpha_0(\delta )}( \data^\delta)
\end{equation} 
is a regularization  method for  \eqref{eq:ip}  on the signal  class $\MM^*$ with respect to the  similarity measure $\dd_{\YY}$.  One calls $\al_0$ an a-prior parameter choice over the set $\MM^*$.  
\end{proposition}

\begin{proof}
For  any $\eps > 0$ choose
$\al(\eps)$ such    $\norm{\Bo_{\al(\eps)}(\data) - \signal} \leq \eps/2$
for  all $\signal \in \MM^*$. .
Moreover, choose $\tau(\eps)$ such that for all $\zsignal \in \YY$ with
$\norm{\data - \zsignal }  \leq  \tau(\eps)$ we have
$\norm{\Bo_{\al(\eps)} (\data) - \Bo_{\al(\eps)}(\zsignal)} \leq \eps / 2$.
Without loss of generality we can assume that $\tau(\eps)$ is strictly increasing and continuous with
$\tau(0+)=0$.  We define  $\al_0 \coloneqq \al \circ \tau^{-1}$. Then, for every $\delta >0$
and $\norm{\data - \data^\delta}  \leq  \delta$, 
\begin{align*} 
\norm{\Bo_{\al_0(\delta)}(\data^\delta) -\signal}
&\leq
\norm{\Bo_{\al_0(\delta)}(\data) -\signal}
+
\norm{\Bo_{\al_0(\delta)}(\data) - \Bo_{\al_0(\delta)}(\data^\delta)} \\
&=
\norm{\Bo_{\al \circ \tau^{-1}(\delta)}(\data) - \signal}
+
\norm{\Bo_{\al \circ \tau^{-1}(\delta)}(\data) - \Bo_{\al \circ \tau^{-1}(\delta)}(\data^\delta)} \\
&\leq
{\tau^{-1}(\delta)}/{2} + {\tau^{-1}(\delta)}/{2}  = \tau^{-1}(\delta)\,.
\end{align*}
Because $\tau^{-1}(\delta) \to 0$ as $\delta \to 0$ this completes the proof.
\end{proof}

A popular  class of regularization methods  is convex variational regularization defined by a convex functional $\reg \colon \XX \to [0, \infty]$.  These methods approximate  right inverses, given by the $\reg$-minimizing solutions of $\Ao (\signal)  = \data$. Such solutions are elements in $\argmin \{\reg(\signal) \mid \signal \in \XX \wedge \Ao (\signal) = \data\}$. Note that  an $\reg$-minimizing solution exists whenever $\XX$ is reflexive, $ \reg$ is coercive and weakly lower semi-continuous, and the  equation $\Ao (\signal)  =  \data $ has at least one solution in the domain of $\reg$. Moreover, $\reg$-minimizing solution are unique if  $ \reg$ is strictly convex. In this case this immediately defines a right inverse for $\Ao$.  Convex variational regularization  is defined by minimizing the Tikhonov functional $
\tik_{\data^\delta, \alpha } \colon \XX \to [0, \infty] \colon 
\signal \mapsto \simm (\Ao (\signal), \data^\delta) 
+ \alpha \reg(\signal) $ for  data $\data^\delta \in \YY$ and regularization parameter $\al >0$.
 In Section~\ref{sec:nett} we will study a more general  form including  non-convex regularizers defined by  a neural network. At his point,  we only state one result on convex variational regularization.

\begin{theorem}[Tikhonov regularization in Banach spaces] \label{thm:var-convex}
Let $\XX$ be reflexive, strictly convex  and  $p, q >  1$.
Moreover, suppose that $\XX$  satisfies the Radon--Riesz property; that is, for any  sequence  $(\signal_k)_{k \in \N} \in \XX^\N$  the weak convergence $\signal_k \rightharpoonup \signal \in \XX$ together with  the convergence in the norm $\norm{\signal_k} \to \norm{\signal}$ implies  $\lim_{k\to \infty} \signal_k =  \signal $ in the norm topology. 
Then the following holds:
\begin{enumerate}
\item  \label{it:var1}
$\Ao^\dag  \colon \ran(\Ao) \to \XX \colon  \data \mapsto \argmin
\set{\norm{\signal} \mid \signal \in \XX \wedge \Ao (\signal) = \data }$ is  well-defined.
\item  
\label{it:var2} For all $ \al >0$ the mapping    
$ \Bo_\al \colon  \YY \to \XX \colon \data^\delta  \mapsto   \argmin \set{ \norm{\Ao (\signal) -\data^\delta }^p+ \norm{\signal}^q \mid \signal \in \XX }$  is well defined and continuous.  
     
\item  \label{it:var4} 
For any $\al_0 \colon (0, \infty) \to 
(0, \infty)$ with  $\al_0\to 0$ and  $\delta^p/\al_0(\delta) \to 0$  as $\delta   \to 0$, the mapping defined by \eqref{eq:reg} and \ref{it:var2}  is a regularization method for \eqref{eq:ip} on $\Ao^\dag(\XX)$ with respect to the norm distance  $\dd_{\XX}$
\end{enumerate}
\end{theorem} 

\begin{proof}
See \cite{ivanov2020theory,scherzer2009variational}. 
\end{proof}

In the Hilbert space setting, the mapping $\Ao^\plus$ defined In Theorem~\ref{thm:var-convex} is given by  the Moore-Penrose inverse, see Proposition \ref{prop:right-hil}   and the text below this Proposition.

\subsection{Deep learning}

In this subsection, we give a brief review of neural networks and deep learning. Deep learning can be characterized as the field where deep neural networks are used to solve various learning problems \cite{lecun2015deep,goodfellow2016deep}.  Several such methods recently appeared as a new paradigm for solving inverse problems.  
In  deep learning literature, neural networks are often formulated in a finite dimensional setting.  To allow a unified treatment we consider  here a general setting  including the  finite dimensional  as well as the infinite dimensional setting.

\begin{problem}[The Supervised learning problem]
Suppose the aim is to find an unknown function $ \NN  \colon  \YY \to \XX$ between two Banach spaces. Similar to classical regression, we are given data $(\data_i,  \signal_i) \in \YY \times \XX$ with   $\NN(\data_i)  \simeq \signal_i$ for $i=1, \dots , N$. From this data, we aim to  estimate the function  $\NN$. For that purpose one chooses a certain  class $(\NN_\theta)_{\theta \in \Theta}$ of functions  $ \NN_\theta  \colon  \YY \to \XX$ and defines $\NN \coloneqq \NN_{\theta^*} $ where $\theta^*$ minimizes the penalized  empirical risk functional    
\begin{equation} \label{eq:risk}
 	\risk_N \colon \Theta \to [0, \infty] \colon 
	\theta \mapsto   
 	\frac{1}{N} \sum_{i=1}^N  L (\NN_\theta (\data_i), \signal_i)  +  P( \theta )  \,.
\end{equation} 
Here  $L  \colon \XX \times \XX \to \R$ is the so called loss function which is a measure  for the error made by the function $\NN_\theta$ on the training examples, and $P$ is  a penalty  that prevents overfitting  of the network and also stabilizes the training process.
\end{problem}

Both, the  numerical minimization of the functional \eqref{eq:risk} and   investigating properties of $\theta^*$ as $N \to \infty$ are of interest in its own \cite{glorot2010understanding,chen2018neural}, but not subject of our analysis.  Instead, most  theory in this chapter is developed under the assumption of  suitable trained prediction functions.

\begin{definition}[Neural network]\label{def:network}
Let $\Theta$  be a parameter set and  $ H_{\ell,\theta} \colon  \XX_0 \times \cdots \times \XX_{\ell-1} \to \XX_\ell$, for $\ell = 1, \dots, L$ and $\theta \in \Theta$, be mappings between Banach spaces with $\XX_0= \Y$ and $\XX_L= \X$.   We call a family $(\NN_\theta)_{\theta \in \Theta}$ of  recursively defined  mappings    
\begin{equation}\label{eq:nn}
	\NN_\theta \coloneqq 
	a_{L,\theta}\colon \YY \to \XX
	\quad \text{ where }  
	\forall \ell \in \set{ 1, \dots, L} \colon a_{\ell,\theta}   
	= H_{\ell,\theta}(\id,a_{1,\theta},\dots, a_{\ell-1,\theta})
\end{equation} 
a neural network.
In that context, $\XX_1, \dots, \XX_{L-1}$ are called the hidden spaces. We  refer to the  individual  members $\NN_\theta$  of a neural network as neural network functions. 
\end{definition}

A neural network in finite dimension can be seen as discretization of  $(\NN_\theta)_{\theta \in \Theta}$ where $\YY$ and $\XX$ are  discretized using any standard discretization approach.         

\begin{example}[Layered neural network]
As a typical example for a neural network consider a layered neural network 
$(\NN_\theta)_{\theta \in \Theta} $ with $L$  layers between finite dimensional spaces.  In this case,  each network function has the form $
 \NN_\theta \colon \R^p \to \R^q\colon 
\data \mapsto
	(\nlo_{L}  \circ \Wo_L^\theta) \circ    \cdots \circ (\nlo_1 \circ \Wo_1^\theta)(\data) $,
where $\Wo_\ell^\theta  \colon \R^{d(\ell-1)}  \to 
\R^{d(\ell)}$ are affine mappings and  $\nlo_{\ell} \colon \R^{d(\ell)}  \to 
\R^{d(\ell)}$ are  nonlinear mappings with $d(0) =p$ and  $d(L)=q$.  The notion indicates that the affine mappings depend on the parameters $\theta \in \Theta$ while the nonlinear mappings are taken fixed. Although this is standard in neural networks, modifications where the nonlinearities contain trainable parameters have been proposed \cite{agostinelli2014learning,ramachandran2017searching}.   The affine parts $\Wo_\ell^\theta$ which are the learned parts  in  the neural network can be  represented by  a $d(\ell)  \times d(\ell-1)$ matrix for the  linear part  and a vector of size $ 1 \times d(\ell)$ for the  translation part. 
\end{example}

In  standard  neural  networks, the entries of the matrix and the bias  vector are taken as independent parameters. For typical inverse problems where the dimensions $p$ and $q$ are large, learning all these numbers  is challenging and perhaps impossible task.  For example, the matrix describing the linear  part of a  layer mapping  a  $200 \times 200$ image  to an image of the same size  already  contains $1.6$ billion parameters. Learning these parameters from data  seems challenging.   Recent neural networks and in particular  convolutional neural networks (CNNs) use the concepts of sparsity and weight sharing to significantly reduce the number of parameters to be learned. 

\begin{example}[CNNs using sparsity and weight sharing]
In order to reduce the number of free parameter between a linear mapping between  images, say of sizes  $q = n\times n$ and $p = n\times n$, CNNs  implement sparsity and weight sharing via convolution operators. In fact,  a 
 convolution operation $\Ko \colon \R^{n\times n} \to \R^{n\times n}$ with kernel size $k \times k$  is represented by $k^2$ numbers which clearly enormously reduces the number $n^4$ of parameters   required to represent a general linear mapping on  $\R^{n\times n}$. To enrich the expressive power of the neural network, actual CNN architectures use multiple-input multiple-output convolutions $\Ko \colon \R^{n\times n \times c} \to \R^{n\times n \times d}$ which uses one convolution kernel for each pair in $\set{1, \dots , c}\times \set{1, \dots , d}$ formed between each input channel and each output channel. This now increases the number of learnable parameters to $cd k^2$ but overall the number of parameters  remains much smaller than for a full dense layer between large images. Moreover, the use of multiple-input multiple-output convolutions  in combination with typical nonlinearities  introduces a flexible and complex structure which demonstrated to give state of the art results in various imaging tasks.
\end{example}

\section{Regularizing networks}
\label{sec:regnet}

Throughout this section let  $\Ao \colon \XX \to \YY$ be a linear forward operator between Banach spaces and $\Bo \colon \ran(\Ao) \to \XX$   a linear right inverse with $\UU \coloneqq  \ran(\Bo)$. In particular, the  kernel of $\Ao$ is complemented and  we can write   $\XX =  \UU \oplus \ker(\Ao)$.     
 The results in this section generalize the methods and some of the results of  \cite{schwab2019deep} from the Hilbert case to the Banach space case.

\subsection{Null-space networks} 

The idea of post-processing networks  is to improve a given right inverse    by applying a network.  Standard networks  however will destroy  data consistency of the initial reconstruction.  Null space networks are the natural class of neural networks restoring data consistency. 

\begin{definition}[Null space network]\label{def:nsn}
We call  the family $(\Id_{\X} + \Po_{\ker(\Ao)}  \circ \nullnet_\theta)_{\theta \in \Theta}$ a null space network if   $(\nullnet_\theta)_{\theta \in \Theta}$  is any network of Lipschitz continuous functions $ \nullnet_\theta  \colon \X \to \X$. We will also  refer to individual functions $\NN_\theta =  \Id_{\X} + \Po_{\ker(\Ao)}  \circ \nullnet_\theta$ as  null space networks.    \end{definition}

Any null space network $\NN_\theta =  \Id_{\X} + \Po_{\ker(\Ao)}  \circ \nullnet_\theta$ preserves data consistency  in the sense that $\Ao (\signal) = \data $ implies  $\Ao (\NN_\theta (\signal)) = \data $, which can be seen from 
\begin{equation} \label{eq:dc}
\Ao \circ \bigl( \Id_{\X} + \Po_{\ker(\Ao)}  \circ \nullnet_\theta \bigr) (\signal) 
= \Ao (\signal)  +  \Ao \circ \Po_{\ker(\Ao)}  \circ \nullnet_\theta (\signal) 
=
\data \,.
\end{equation} 
A standard residual network $\Id_\XX + \nullnet_\theta$ often used as post processing network in general does not satisfy this  property \eqref{eq:dc}.

\begin{psfrags}
\psfrag{x}{$\signal$}
\psfrag{s}{$\operatorname{ReLU}$}
\psfrag{D}{$\nullnet_\theta(\signal)$}
\psfrag{R}{$\signal  + \nullnet_\theta (\signal) $}
\psfrag{i}{identity layer}
\psfrag{w}{weight layer}
\psfrag{p}{projection layer}
\psfrag{P}{$ \Po_{\ker(\Ao)} \circ \nullnet_\theta(\signal)$}
\psfrag{N}{$\signal  + \Po_{\ker(\Ao)}  \circ \nullnet_\theta  (\signal) $}

\begin{figure}[htb!]
\begin{center}
  \includegraphics[width=0.8\columnwidth]{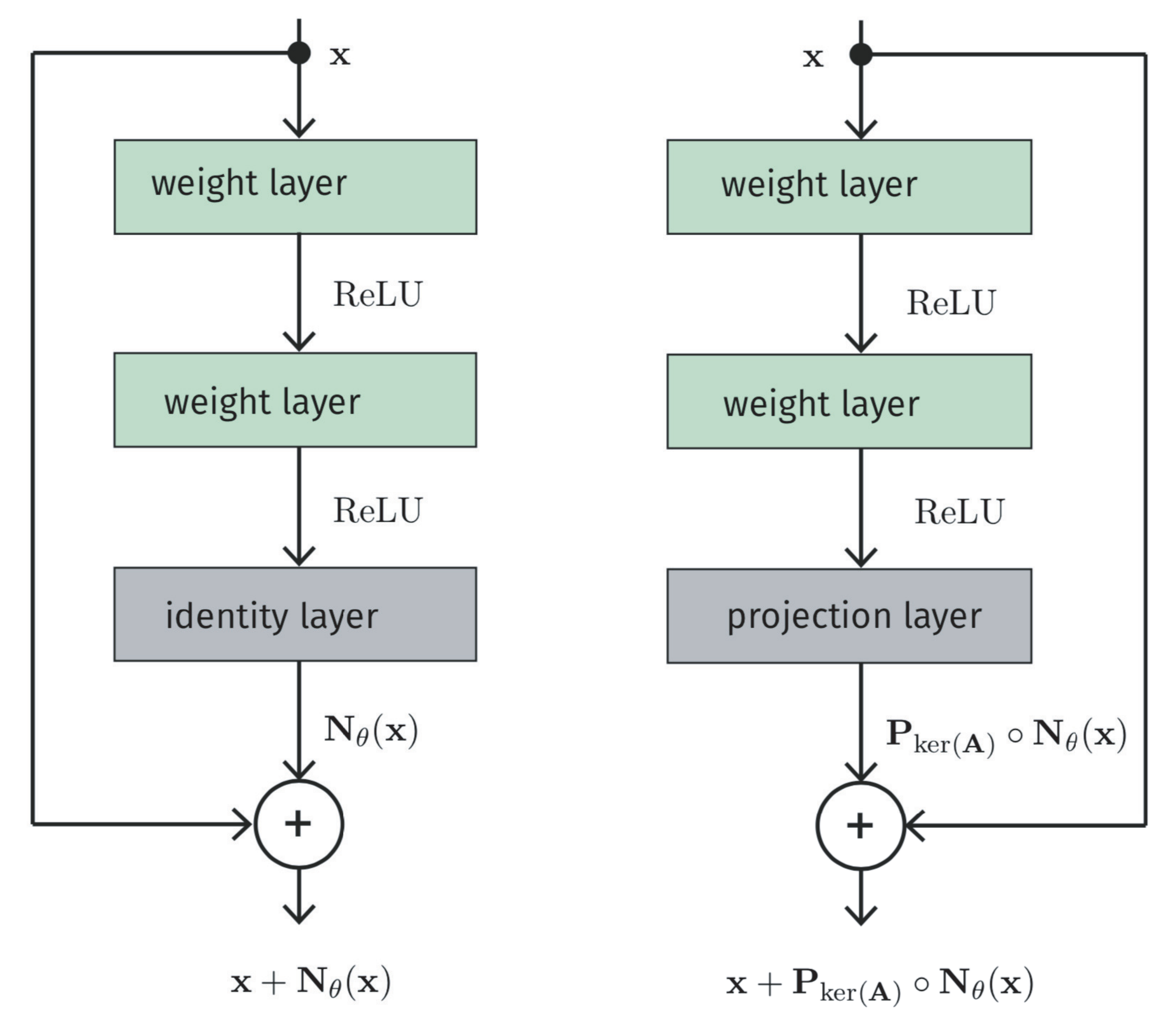}
\caption{\label{fig:RNnet} Residual network $\id  + \nullnet_\theta$ (left) versus null space network $\id  + \Po_{\ker(\Ao)}  \circ \nullnet_\theta$ (right). The difference between the two architectures  is the projection layer $\Po_{\ker(\Ao)} $ in the null space network after the last weight layer.}
\end{center}
\end{figure}
\end{psfrags}

\begin{remark}[Computation of the projection layer]
A main ingredient in the null-space network is the computation of the   
projection layer $\Po_{\ker(\Ao)}$. In some cases, it can be computed explicitly. For example, if $\Ao = \So_I \circ \Fo$ is the subsampled Fourier transform, then $\Po_{\ker(\Ao)} = \Fo^* \circ \So_I \circ \Fo$. For a general forward operator between Hilbert spaces the projection  $\Po_{\ker(\Ao)} \zz$  can be implemented via standard methods for solving linear equation. For example, using the starting value $\zz$  and solving the equation $\Ao (\signal) = 0$ with the CG (conjugate gradient) method for the normal equation or Landwebers methods gives a sequence that converges to the projection  
 $\Po_{\ker(\Ao)} \zz = \argmin\set{ \norm{\signal - \zz}  \mid  \Ao (\signal) = 0} $.\end{remark}

An example comparing a standard  residual network $\Id_\XX + \nullnet_\theta$ and a null space network $\Id + \Po_{\ker(\Ao)} \circ \nullnet_\theta$ both with two weight layers are shown in Figure~\ref{fig:RNnet}.

\begin{proposition}[Right inverses defined by  null space networks]\label{prop:null}
Let $\Bo\colon \YY \to \XX$  be a given linear  right inverse and $ \NN_\theta = \Id_{\X} + \Po_{\ker(\Ao)}  \circ \nullnet_\theta$ be a null space network. Then the composition 
\begin{equation}\label{eq:nullnet}
\NN_\theta \circ \Bo \colon \ran(\Ao) \to \XX \colon
\data \mapsto (\Id_{\X} + \Po_{\ker(\Ao)}  \circ \nullnet_\theta) ( \Bo \data) 
\end{equation} 
is  right inverse of $\Ao$. Moreover, the following assertions are equivalent:  
\begin{enumerate} [label=(\roman*)]
\item\label{prop:null1}  $ \NN_\theta \circ \Bo$  is continuous 
\item\label{prop:null2}  $ \Bo$  is continuous 
\item\label{prop:null3} $ \ran(\Ao)$  is closed.     
 \end{enumerate}
\end{proposition}

\begin{proof}
Because $\Bo$ is a right inverse we have $(\Ao \circ \Bo)( \signal) = \data $ for all $\data \in \ran (\Ao)$. Hence, the data consistence property \eqref{eq:dc} implies $ \Ao (  ( (\Id_{\X} + \Po_{\ker(\Ao)}\circ \nullnet_\theta  ) \circ \Bo) (\signal)) =  \data$, showing that $ \NN_\theta \circ \Bo$ is a right inverse  of $\Ao$. The implication \ref{prop:null1} $\Rightarrow$ \ref{prop:null2} 
follows from the identity $\Po_{\UU} \circ \NN_\theta \circ \Bo = \Bo$
and the continuity of the projection. The   implication \ref{prop:null2} $\Rightarrow$ \ref{prop:null3} follows from the continuity of $\NN_\theta$. Finally, the equivalence \ref{prop:null2} $\Leftrightarrow$ \ref{prop:null3}         
has been established in Proposition~\ref{prop:rightlin}. 
\end{proof}

A benefit of non-linear  right inverses  defined by null space networks 
is that  they can be adjusted  to a given image class.  A possible network training  is given as follows

\begin{remark}[Possible training strategy]
The null space  network $ \NN_\theta  =  \Id + \Po_{\ker(\Ao)} \circ \nullnet_\theta$
can be trained to map elements in $\MM$ to the elements from the  desired class of images. For  that purpose, select   
training data pairs $ (\signal_1, \zsignal_1 ), \dots, (\signal_N, \zsignal_N ) $ with  $\zsignal_i =  \Bo \circ \Ao (\signal_i)$ and minimize the regularized empirical risk, 
\begin{equation} \label{eq:err1}
\risk_N \colon \Theta  \to [0, \infty]
    \colon \theta  \mapsto  \frac{1}{N} \sum_{n=1}^N  
    \norm{ \signal_i - \NN_\theta ( \zsignal_i )  }^2
    + \norm{ \theta}_2^2 \,. 
\end{equation}
Note that for our analysis it not required that  \eqref{eq:err1} is exactly minimized. Any  null space network $ \NN_\theta$  where $\sum_{n=1}^N  
    \norm{ \signal_i - \NN_\theta ( \zsignal_i )  }^2$ is small
yields  right inverse  $ \NN_\theta \Bo $ does a better job in estimating $\signal_n$ from data $ \Ao  \signal_n$ than the original right inverse $\Bo$.
\end{remark}

Proposition \ref{prop:null} implies that the  solution of  ill-posed problems by null space networks requires the use of stabilization methods similar to the case of classical methods.  In the following subsection, we show that a the combination of null-space network  with a regularization of $\Bo$  in fact yields a regularization method on a signal class related  to the null-space network.

\subsection{Convergence analysis}

Throughout the following, let  $ \NN_\theta  =  \Id + \Po_{\ker(\Ao)} \circ \nullnet_\theta$ be null-space network and  $(\Ro_\delta)_{\delta >0}$ be a  regularization method for \eqref{eq:ip}  on the signal  class   $\UU^* \subseteq \UU$ with respect to the  similarity measure $\simm$ as introduced   in 
Definition~\ref{def:reg}.
As as illustrated in Figure~\ref{fig:null}, we  consider the family  $(\NN_\theta \circ \Ro_\delta )_{\delta >0}$ of compositions of the regularization method with the null space network.

\begin{figure}[htb!]
\begin{tikzpicture}[scale=1]
\draw[line width=1.3pt,->] (0,0)--(7,-3);
\node at (8.2,-3.4){$\UU^* \subseteq  \ran(\Bo)$};
\draw[line width=1.3pt,->] (3,-2)--(4.5,1.5);
\node at (4.5,1.8){$\ker(\Ao)$};
\draw [line width =0.6mm, lblue] plot [smooth] coordinates {(0.5,7/6) (3,1.4) (4,0.5) (5,1) (7,0.7) (8,1)};
\node at (0.5,8.3/6)[above] {$\MM^* = \NN_\theta ( \UU^*)$};
\filldraw (5,-15/7)circle(1.7pt);
\node at (4.4,-31/14)[below]{$\Ro_\delta(\data)$};
\draw [dashed, line width=1.3pt,->] (5,-15/7)--(5+2.9*3/7,-15/7+2.9*1);
\filldraw[color=black] (5+2.95*3/7,-15/7+2.95*1)circle(3pt);
\node at (5+10/7,-15/7+3.1*1)[above]{$(\NN_\theta \circ \Ro_\delta)( \data)$};
\node at (-3,0){};
\end{tikzpicture}
\caption{\label{fig:null} Linear Regularization $(\Ro_\delta)_{\delta >0}$ combined with a null space network $ \NN_\theta  =  \Id + \Po_{\ker(\Ao)} \circ \nullnet_\theta$.
We start  with a linear regularization  $\Ro_\delta \data$ and  the nullspace network $\NN_\theta  = \id + \Po_{\ker(\Ao)} \circ \nullnet_\theta $  adds  missing parts along the null space $\ker(\Ao)$.}
\end{figure}

\begin{theorem}[Regularizing null-space network]\label{thm:conv}
 For a given null-space network  $ \NN_\theta  =  \Id + \Po_{\ker(\Ao)} \circ \nullnet_\theta$ and  a given  regularization method $(\Ro_\delta)_{\delta >0}$ on the signal  class   $\UU^*$, the   family  $(\NN_\theta \circ \Ro_\delta )_{\delta >0}$ is regularization method for \eqref{eq:ip}  on the signal  class   $\NN_\theta(\UU^*)$ with respect to the  similarity measure $\simm$. 
We call $(\NN_\theta \circ \Ro_\delta )_{\delta >0}$ a regularizing null-space network.
\end{theorem}

\begin{proof}
Let $L$ be a Lipschitz constant of $\NN_\theta$ and recall $\NN_\theta = \Id + \Po_{\ker(\Ao)} \circ \nullnet_\theta$. For any $\signal \in  \NN_\theta(\UU^*) $ and  $ \data^\delta \in \YY$ we have
 \begin{multline*}
 \norm{ \signal -   \NN_\theta \circ \Ro_\delta (\data^\delta)}
=
\norm{   (\Id + \Po_{\ker(\Ao)} \circ \nullnet_\theta ) (\Bo \circ \Ao (\signal)) -  ( \Id + \Po_{\ker(\Ao)} \circ \nullnet_\theta )  (\Ro_\delta (\data^\delta))}
\\
\leq
L  \norm{ (\Bo \circ \Ao) (\signal) -   \Ro_\delta (\data^\delta)} \,.
 \end{multline*}
Here  we have uses  the identity $\signal = (\Id + \Po_{\ker(\Ao)} \circ \nullnet_\theta ) ((\Bo \circ \Ao) (\signal))$ 
 for   $\signal \in  \ran(\NN_\theta) $. 
 Consequently
 \begin{multline*}
 \sup \set{\norm{ \signal -   (\NN_\theta \circ \Ro_\delta) (\data^\delta)} \mid \data^\delta \in \YY \wedge \simm(\data^\delta, y) \leq \delta}
 \\ \leq
 L   \sup \set{\snorm{ (\Bo \circ \Ao) (\signal)   -      \Ro_\delta (\data^\delta) } \mid \data^\delta \in \YY \wedge  \simm(\data^\delta, y)  \leq \delta} 
  \to 0 \,.
 \end{multline*}
 In particular,  $(\NN_\theta \circ \Ro_\delta )_{\delta >0}$ is regularization method for \eqref{eq:ip}  on   
 $\NN_\theta(\UU^*)$ with respect to the  similarity measure $\simm$
\end{proof}

In Hilbert spaces a wide class of regularizing reconstruction networks can be defined by
regularizing filters.

\begin{example}[Regularizations defined by filters]
Let $\XX$ and $\YY$  be  Hilbert spaces. A family $\kl{g_\alpha}_{\al >0}$  of functions
$g_\al \colon [0,\norm{\Ao^*\circ\Ao}]  \to \R $
 is called a regularizing   filter   if it satisfies
 \begin{itemize}
\item
For all $\al >0$,  $g_\al$ is piecewise continuous;
\item
 $\exists C  >0 \colon  \sup \set{ \sabs{\la g_\al \skl{\la}} \mid \al >0  \wedge
\la \in  [0,\norm{\Ao^*\circ\Ao}] } \leq C$.
\item
$\forall \la \in (0,\norm{\Ao^*\circ\Ao}] \colon
\lim_{\al \to 0} g_\al \skl{\la} = 1/\la$.
\end{itemize}
For a given regularizing filter $\kl{g_\alpha}_{\al >0}$ define
$\Bo_\al \coloneqq g_{\al} \kl{ \Ao^*\circ\Ao} \Ao^*$.
Then for a suitable  parameter  choice $\al = \al(\delta, \data)$ the  family $(\Bo_{\al(\delta, \edot) })_{\delta > 0}$ is a regularization method on  $\ran(\Ao^{\plus})$.
Therefore,  according to Theorem~\ref{thm:conv},
the family $(\nun_\theta \circ \Bo_{\al(\delta, \edot) })_{\delta >0}$ 
is a regularization method  on $\nun_\theta  ( \ran(\Ao^{\plus}))$.
Note that in  this setting one  can  derive derive  quantitative error estimates (convergence rates); we the refer the interested reader to the original paper \cite{schwab2019deep}.      
\end{example}

\subsection{Extensions}

The regularizing null-space networks defined in Theorem~\ref{thm:conv} 
are of the form  $\NN_\theta \circ \Ro_\delta$ where $\Ro_\delta$ is a classical regularization and $\NN_\theta$ only  acts in null space of $\Ao$. In  order to better accounts for noise, it is beneficial  allowing the  networks to modify  $\Ro_\delta$ also on the complement $\UU$.       

\begin{definition}[Regularizing family of networks] \label{def:regnets}
Let  $(\Ro_\delta)_{\delta >0}$ be a  regularization method for \eqref{eq:ip}  on the signal  class   $\UU^* \subseteq \UU$ with respect to the  similarity measure $\simm$ as introduced  in  Definition~\ref{def:reg}.
A family  $(\NN_{\theta(\delta)} \circ \Ro_\delta )_{\delta >0}$ is  called regularizing family of networks if $(\NN_{\theta})_{\theta \in \Theta}$ is a neural network such that  the network functions $\NN_{\theta(\delta)}\colon \XX \to \XX$, for $\delta >0$,  are  uniformly Lipschitz continuous  and 
\begin{equation*}
\forall \zz \in \ran(\Bo)  \colon \quad 
\lim_{\delta \to  0} \NN_{\theta(\delta)} (\Ro_\delta \circ \Ao (\zsignal)) = \nullnet(\zsignal)  \,,
\end{equation*}
for some null-space network $\nullnet $.
\end{definition}

Regularizing families of networks have been introduced in \cite{schwab2020big} where it has been shown that  a regularizing family of networks defines a regularization method together with convergence rates. Moreover, an example in the form of a data-driven extension of  truncated SVD regularization has been given.    
In a finite  dimensional setting,  related extension of null-space networks named deep decomposition learning  has been introduced in  \cite{chen2019deep}. A combination  of null-space learning with shearlet reconstruction for limited angle tomography has been introduced in \cite{bubba2019learning}.  In \cite{dittmer2019projectional}, a neural network based projection approach based on approximate  data consistency  sets  has been studied. Relaxed versions  of null-space networks, where approximate data consistency is  incorporated  via   a confidence region  or a  soft penalty are proposed in \cite{huang2020data,kofler2019neural}. Finally, extensions  of the null-space approach to non-linear problems are studied  in \cite{boink2020data}.

\section{The NETT approach}
\label{sec:nett}

Let us recall that  convex variational regularization of the inverse problem \eqref{eq:ip} consists in minimizing the generalized Tikhonov functional $
\simm (\Ao (\signal), \data^\delta)  + \alpha \reg(\signal)$, where $\reg$ is a convex functional and $\simm$ a similarity measure (see Section ~\ref{sec:reg}). The regularization term $\reg$ is traditionally a semi-norm defined on a dense subspace  of $\XX$. In this section, we will extend this setup by using deep learning techniques with learned regularization functionals. 

\subsection{Learned regularization functionals}

We assume that the regularizer takes the  form
\begin{equation} 	\label{eq:cnn}
\forall \signal \in \XX \colon
	\quad
	\reg(\signal)
	=
	\reg_\theta(\signal)
	\coloneqq
	\nlf_\theta(\NN_\theta(\signal)) \,.
\end{equation}
Here $\nlf_\theta \colon  \XXX  \to [0, \infty]$  is a scalar
functional and $\NN_\theta(\edot) \colon \XX \to \XXX$   a  neural  network where $\theta \in \Theta$, for some vector space $\Theta$ containing free parameters that can be adjusted by available training data. From the representation learning point of view \cite{bengio2013representation}, $\NN_\theta(\signal)$ can be interpreted as a learned representation of $\signal$. It could be constructed in such a way that $\psi_\theta \circ \NN_\theta$ is minimal for a low dimensional manifold where  the 
true signals $\signal$  are clustered around. Finding such manifold for biomedical images has been an active research topic on manifold learning \cite{georg2008manifold,wachinger2012manifold}. Deep learning has also been used for this purpose \cite{brosch2013manifold}. A learned regularizer $\reg_\theta = \psi_\theta \circ \NN_\theta$ reflects the statistics of the signal space, which penalizes those who deviate from the data manifold.

 The similarity measure is taken as $\simm_\theta \colon \Cone \times \Cone \to [0, \infty] $, where $\Cone$ is a conic closed subset in $\YY$. It is not necessarily symmetric in its arguments. One may take $\simm_\theta(\Ao(\signal), \data)$ to be common hard-coded consistency measure such as $\simm(\Ao(\signal), \data) = \snorm{\Ao(\signal) - \data}^2$  or the Kullback-Leibler divergence (which, among others, is used in emission tomography). On the other hand, it can be a learned measure, defined via a neural network.
A learned consistency measure $\simm_\theta$ reflects the statistics in the data (measurement) space. 
It can learn to reduce uncertainty in the data measurement process, e.g. by identifying non-functional transducers.  It can also learn to reduce the error in the forward model \cite{aljadaany2019douglas}.  Finally, it may encode the range description of the forward operator, which has not been successfully exploited in inverse problems by traditional methods. In summary, it can be said that learned consistency measures have potentially high impact in solving inverse problems.

Using the neural network based learned regularizer \eqref{eq:cnn} and a learned discrepancy measure as discussed above, results in the following optimization problem 
\begin{equation}\label{eq:nett}
\argmin_{\signal \in \DD} \, \tik_\theta (\signal)\coloneqq\simm_\theta(\Ao(\signal), \data)
 + \alpha \reg_\theta(\signal) \,.  
 \end{equation}
Solving \eqref{eq:nett}  can  be seen as a neural networks-based variant of generalized Tikhonov regularization for solving \eqref{eq:ip}. Following \cite{li2018nett} we therefore call \eqref{eq:nett} the network Tikhonov (NETT) approach for solving  inverse problems.  Currently, there are two main approaches for integrating neural networks in the NETT approach  \eqref{eq:nett}: (T1) training the neural networks simultaneously with solving the optimization problem, and (T2) training the network independently before solving the optimization problem.

Approach (T1) fuses the data with a solution method of the optimization problem \eqref{eq:nett}. The resulting neural networks, therefore, depend on the method to solve the optimization. This approach enforces the neural networks to learn particular representations that are useful for the chosen optimization technique. These representations will be called \emph{solver-dependent}. The biggest advantage of this end-to-end approach is to provide a direct and relatively fast solution $\signal$ for given new data $\data$. It is commonly realized by unrolling an iterative process \cite{arridge2019solving}. The  resulting neural network is a cascade of relatively small neural networks, each of them is, possibly a variant of, those appearing in the data consistency or regularization term. It is worth noting that the neural network does not aim for representation learning. Each layer or block serves to move the approximate solution closer to the exact solution.  In contrast to typical iterative methods, each block in a unrolled neural network can be different from others.  This is explained  to speed up the convergence of the learned iterative method.  The success of this approach is an interesting phenomenon that needs further investigation. The use of neural networks to implement and accelerate iterative methods to solve traditional regularization methods has been intensively studied. We refer the reader to \cite{arridge2019solving} and the references contained therein.

The approach (T2) is more modular \cite{li2018nett,Lunz2018} and results in smaller training problems and is closer to the meaning of representation learning. The training of the regularizer may or may not depend on the forward operator $\Ao$. In the former case, the resulting representation is called \emph{model-dependent} while the latter is \emph{model-independent}. Model-dependent representation seems to be crucial in inverse problem for two reasons. The first reason is that it aligns with the inverse problem (and better serves any solution approach). Secondly, in medical imaging applications, the training signals are often not the ground truth signals. They are normally obtained with a reconstruction method from high quality data. Therefore, while training the regularizer, one should also keep in mind the reconstruction mechanism of the training data. A possible approach is to first train a baseline neural network to learn model-independent representation. Then an additional block is added on top to train for model-dependent representation. This has been shown in \cite{obmann2020sparse} to be a very efficient strategy. 

Let us mention that approach (T1) has richer literature than (T2), but less (convergence) analysis. In this section, we focus more on (T2), where we establish the convergence analysis and convergence rate in Section~\ref{S:b}. This is an extension of our works \cite{2019arXiv190803006H,obmann2020sparse}. In Section~\ref{S:a}, we review a few existing methods that are most relevant to our discussion, including some works in approach (T1). We also propose INDIE, which can be regarded as an operator inversion-free variant of the MODL technique \cite{aggarwal2018modl} and can make better use of parallel computation.

\subsection{Convergence analysis} \label{S:b}

Analysis for regularization with neural networks has been studied in \cite{li2018nett} and \cite{2019arXiv190803006H}.  In this section, we further investigate the issue. To this end, we consider the approach (T2), where the neural networks are trained independently of the optimization problem \eqref{eq:nett}. That is, $\theta = \theta^*$ is already fixed a-priori. For the sake of simplicity, we will drop $\theta$ from the notation of $\reg_\theta$ and $\simm_\theta$. We focus on how the problem depends on the regularization parameter $\alpha$ and noise level $\delta$ in the data. Such analysis in standard situations is well-studied, see, e.g., \cite{scherzer2009variational}. However, we need to extend the analysis to more general cases to accommodate the fact that $\reg$ comes from a neural network and is likely non-convex. 

Let us make several assumptions on the regularizer and fidelity term.

\begin{condition}
\label{cond:main} \mbox{}
\begin{enumerate}[leftmargin=3em,label = (A\arabic*)]

 \item \label{cond:main1} Network regularizer $\reg$: $\XX \to [0,\infty]$ satisfies
 \begin{itemize}[wide]
\item[(a)] $0 \in \Dom(\reg):= \{ \signal \mid  \reg(\signal)< \infty\}$;
\item[(b)] $\reg$ is lower semi-continuous;
\item[(c)] $\reg(\edot)$ is coercive, that is $\reg(\signal) \to \infty$
as $\norm{\signal} \to \infty$.

\end{itemize}

\item \label{cond:main2}
 Data  consistency term $\simm$: $\Cone \times \Cone \to [0,\infty]$ satisfies
 \begin{enumerate}[wide]
\item $\Dom (\simm(0,\cdot)) = \Cone$;
\item If $\simm(\data_0, \data_1) < \infty$ and $\simm(\data_1,\data_2)<\infty$ then $\simm(\data_0,\data_2)< \infty$; 
\item $\simm(\data, \data') = 0 \iff \data = \data'$; 
\item $\simm(\data, \data')\geq C \|\data - \data'\|^2$ holds in any bounded subset of $\Dom(\simm)$; 
\item For any $\data$, the function $\simm(\data, \edot)$ is continuous and coercive on its domain;
\item The functional $(\signal, \data) \mapsto \simm(\Ao(\signal), \data)$ is sequentially lower semi-continuous in the weak topology of $\XX$ and strong topology of $\YY$. 
\end{enumerate}

\end{enumerate}
\end{condition}

For \ref{cond:main1}, the coercivity condition (c) is the most restrictive. However, it can be accommodated. One such regularizer is proposed in our recent work \cite{2019arXiv190803006H} as follows
\begin{equation} \label{eq:sn}
\reg (\signal) =  \phi( \decoder ( \signal )  ) +
\frac{\beta}{2}\norm{\signal - (\ddecoder \circ  \decoder) ( \signal )}_2^2 \,.
\end{equation}
Here, $ \ddecoder \circ \decoder  \colon  \X \to \X $ is an encoder-decoder network.  The regularizer $\reg$ is to enforce that a reasonable solution $\signal$ satisfies $\signal \simeq  (\ddecoder \circ \decoder) ( \signal )$ and $\phi( \decoder ( \signal ) )$ is small.
The  term $\phi( \decoder ( \signal )  )$ implements  learned prior knowledge, which is normally a sparsity measure in a non-linear basis.
The second term $\norm{\signal - (\ddecoder \circ  \decoder) ( \signal )}_2^2$ forces $\signal$ to be close to data manifold  $\M$. Their combination also guarantees the coercivity of the regularization functional $\reg$.
Another choice for $\reg$ was suggested in \cite{li2018nett}.

For \ref{cond:main2}, $\Cone$ is a conic set in $\YY$. For any $\data \in \Cone$, we define $\Dom(\simm(\data,\edot)) = \{\data' \mid  \simm(\data, \data') < \infty\}$.  The data consistency conditions in \ref{cond:main2} are flexible enough to be satisfied by a few interesting cases. The first example is that $\simm(\data, \data') = \| \data -\data\|^2$, which is probably the most popular data consistency measure. Another case is the Kullback-Leibler divergence, which reads as follows. Let $\YY = \R^n$ and $\Ao: \XX \to \YY$ is a bounded linear positive operator.\footnote{$\Ao$ is positive if: $\data \geq  0 \Rightarrow \Ao \data \geq 0$.} Consider nonnegative cone $\Cone = \{(\data_1,\dots,\data_n) \mid  \forall i \colon \data_i \geq 0\}$. We define $\simm\colon  \Cone \times \Cone \to [0,\infty]$ by
$$\simm(\data, \data')  = \sum_{i=1}^n \data_i \log \frac{\data_i}{\data'_i} + \data'_i - \data_i.$$
It is straight forward to check that Condition~\ref{cond:main2} is satisfied in this case. In particular, item (d) has been verified in  \cite[Equation (13)] {resmerita2007joint}.

To emphasize the fact that our data is the noisy version $\data^\delta$ of $\data$, we rewrite \eqref{eq:nett} as follows 
\begin{equation} \label{eq:nett1}
\argmin_{\signal \in \DD} \, \tik_{\data^\delta,\alpha} (\signal)\coloneqq\simm(\Ao(\signal), \data^\delta)
 + \alpha \reg(\signal). 
\end{equation}
Here, $\DD$ is a weakly-closed conic set in $\XX$ such that $\A(\DD) \subseteq \Cone$.

\begin{theorem}[Well-posedness and convergence]\label{thm:well} 
Let Condition~\ref{cond:main} be satisfied.
Then the following assertions hold true:
\begin{enumerate}
\item \label{thm:well1}
\emph{Existence:} For all $\data \in \Cone$ and $\alpha >0$, there
exists a minimizer of $\tik_{\alpha;\data}$ in $\DD$.
\item \label{thm:well2}
\emph{Stability:} If $\data_k \to \data$, $\simm(\data,\data_k)< \infty$ and $\signal_k \in \argmin  \tik_{\alpha; \data_k}$,
then weak accumulation points of $(\signal_k)_{k \in \N}$ exist and are minimizers of $\tik_{\alpha;\data}$.

\item\label{thm:well3}
\emph{Convergence:}  Let $\data \in \ran(\Ao) \cap \Cone$ and
 $(\data_k)_{k \in \N}$ satisfy $\simm(\data, \data_k) \le \delta_k$ for some sequence $(\delta_k)_{k\in \N} \in (0, \infty)^\N$ with
  $\delta_k \to 0$. Suppose $\signal_k \in \argmin_\signal\tik_{\data_k, \alpha(\delta_k)}(\signal)$, and let the  parameter choice $\alpha \colon (0, \infty) \to (0, \infty) $ satisfy
\begin{equation} \label{eq:alpha}
\lim_{\delta \to 0} \alpha(\delta) = \lim_{\delta \to 0} \frac{\delta}{\alpha(\delta)} = 0 \,.
\end{equation}
Then the following holds:
\begin{enumerate}[label= (\arabic*)]
\item\label{thm:well31}
All weak accumulation points of $(\signal_k)_{k \in \N}$ are $\reg$-minimizing solutions of the equation 
$\Ao(\signal) = \data$;
\item\label{thm:well32}
$(\signal_k)_{k \in \N}$ has at least one weak  accumulation point $\signal^\plus$;
\item \label{thm:well33}
Every  subsequence $(\signal_{k(n)})_{n \in \N}$ that weakly converges to
$\signal^\plus$ satisfies  $\reg(\signal_{k(n)}) \to \reg(\signal^\plus)$;
\item\label{thm:well34}
If the $\reg$-minimizing solution of~$\Ao(\signal) = \data$ is unique, then $\signal_k \rightharpoonup \signal^\plus$.
\end{enumerate}
\end{enumerate}
\end{theorem}

Before starting the proof, we recall that $\signal^\plus$ is an $\reg $-minimizing solution of
the equation $\Ao  \signal = \data$ if $\signal^\plus \in \argmin \set{\reg (\signal ) \mid \signal \in \DD \wedge \Ao  \signal = \data}$. 

\begin{proof}
\ref{thm:well1} Firstly, we observe that $c:= \inf_\signal \tik_{\data, \alpha}(\signal) \leq \tik_{\data, \alpha}(0) < \infty$. Let $(\signal_k)_k$ be a sequence such that $\tik_{\data, \alpha}(\signal_k) \to c $. There exists $M>0$ such that $\tik_{\data, \alpha}(\signal_k) \leq M$, which implies $\alpha \reg(\signal_k) \leq M$. Since $\reg$ is coercive, we obtain $(\signal_k)_k$ is bounded. By passing into a subsequence, $\signal_{k_i} \rightharpoonup \signal^* \in \DD$. Due to the lower semi-continuity of $\tik_{\alpha,\edot}(\edot)$, we have $\signal^* \in \argmin \tik_{\data, \alpha}$.

\ref{thm:well2}  Since $\signal_k \in \argmin \tik_{\data, \alpha}$, it holds  $\tik_{\data_k, \alpha}(\signal_k)  \leq \tik_{\data_k, \alpha}(0)  = \simm(0, \data_k) + \alpha \reg(0)$. 
Thanks to the continuity of  $\simm(0, \edot)$ on $\Cone$, $(\simm(0,\data_k))_k$ is a bounded sequence. Therefore, 
$\alpha \reg(\signal_k) \leq \tik_{\data_k, \alpha}(\signal_k)  \leq M,$ for a constant $M$ independent of $k$. Since $\reg$ is coercive, $(\signal_k)_k$ is bounded and hence has a weakly-convergent subsequence $\signal_{k_i} \rightharpoonup \signal^\plus$. 

Let us now prove that $\signal^\plus$ is a minimizer of $\tik_{\data, \alpha}$. Since $\tik_{\data,\alpha}(\signal)$ is lower semi-continuous in $\signal$ and $\data$,  \begin{equation} \label{E:lowb} \liminf_{k_i  \to \infty} \tik_{\data_{k_i}, \alpha}(\signal_{k_i}) \geq \tik_{\data, \alpha}(\signal^\plus). \end{equation}
On the other hand, let $\signal \in \DD$ be such that $\tik_{\data, \alpha}(\signal) <\infty$. We obtain $\simm(\Ao(\signal), \data) <\infty$ and $\reg(\signal)< \infty$. Condition \ref{cond:main2}(d) and $\simm(\data,\data_k)<\infty$ give $\simm(\Ao(\signal), \data_k) < \infty$. That is, $\data_k \in \Dom(\simm(\Ao(\signal),\cdot)$. The continuity of $\simm(\Ao(\signal), \cdot)$ on its domain implies $\simm(\Ao(\signal),\data_k) \to \simm(\Ao(\signal), \data)$.  Since $\signal_k$ is the minimizer of $\tik_{\data_k, \alpha}$, $\tik_{\data_k, \alpha}(\signal_k) \leq \tik_{\data_k, \alpha}(\signal)$. Taking the limit, we obtain  $\limsup_k \tik_{\data_k, \alpha}(\signal_k) \leq \tik_{\data, \alpha}(\signal)$.
From \eqref{E:lowb}, $\tik_{\data, \alpha}(\signal^\plus) \leq \tik_{\data, \alpha}(\signal)$ for any $\signal \in \DD$. We conclude that $\signal^\plus$ is a minimizer of $\tik_{\data, \alpha}$.

\ref{thm:well3} We prove the properties item by item. 
\begin{enumerate}[label= (\arabic*)]
\item  Since $\data \in \R(\A)$, we can pick be a solution $\bar \signal$ of $\Ao(\signal) = \data$.  We have
\begin{equation} \label{eq:min} \simm(\A(\signal_k), \data_k) + \alpha_k \reg(\signal_k)  \leq \simm(\data, \data_k) + \alpha_k \reg(\bar \signal)  \leq \delta_k +  \alpha_k \reg(\bar \signal).\end{equation}
Assume that $\signal^\plus$ is a weak accumulation point of $\signal_k$, then 
$$\simm(\A(\signal^\plus), \data) \leq \lim_{k \to \infty} \inf  \simm(\A(\signal_k), \data_k)  \leq \lim_{k \to \infty} \inf (\delta_k +  \alpha_k \reg(\bar \signal)) =0.$$
Therefore, $\simm(\A(\signal^\plus), \data) =0$ or $\A(\signal^\plus) = \data$. Moreover, $\reg(\signal_k)  \leq  \delta_k/\alpha_k + \reg(\bar \signal)$, which implies
$\reg(\signal^\plus) \leq \liminf  \reg(\signal_k)  \leq  \reg(\bar \signal)$.
Since this holds for all possible solution $\bar \signal$ of $\Ao(\signal) = \data$, we conclude that $\signal^\plus$ is a $\reg$-minimizing solution of $\Ao(\signal) = \data$. 

\item Using again the inequality $\reg(\signal_k)  \leq \delta_k/\alpha_k + \reg(\bar \signal)$ and $\reg$ is coercive, we obtain $\{\signal_k\}$ is bounded. Therefore, $\{\signal_k\}$ has a weak accumulation point $\signal^\plus$. 

\item Using \eqref{eq:min} again for $\bar \signal = \signal^\plus$,  we obtain
$\reg(\signal_k)  \leq \delta_k/\alpha_k + \reg(\signal^\plus)$, which gives
$\limsup_k \reg(\signal_k)  \leq \reg(\signal^\plus)$. This together with the fact that $\reg$ is lower semi-continuous gives $\reg(\signal_{k(n)}) \to \reg(\signal^\plus)$.

\item 
The last conclusion follows straight forwardly from the above three.\qedhere \end{enumerate}
\end{proof}

Let us proceed to obtain some convergence results in the norm. Following \cite{li2018nett}, we introduce the absolute Bregman distance.

\begin{definition}[Absolute Bregman distance] \label{def:bregman}
Let $\func  \colon  \DD \subseteq \XX \to \R$
be G\^ateaux differentiable at $\signal \in \XX$.
The \emph{absolute Bregman distance}
$\B_{\func}(\edot, \signal) \colon  \DD  \to [0, \infty]$
with respect to $\func$ at  $\signal$ is defined by
\begin{equation}\label{eq:abreg} 
\forall \tilde\signal \in \XX \colon \quad
\B_{\func}(\tilde\signal, \signal) \coloneqq \abs{\func(\tilde\signal) - \func(\signal) - \func'(\signal)(\tilde\signal - \signal)} \,.
\end{equation}
Here $\func'(\signal)$ denotes the G\^ateaux derivative of $\func$ at $\signal$.
\end{definition}

From Theorem~\ref{thm:well}  we can conclude convergence of $\signal_\alpha^\delta$
to the exact solution in the absolute Bregman distance $\B_\reg$. Below we show that this
implies strong convergence under some additional assumption on the regularization
functional. For this purpose, we define the concept of total non-linearity,
which was introduce in \cite{li2018nett}.

\begin{definition}[Total non-linearity]\label{def:total}
Let $\func \colon  \DD \subseteq \XX \to \R$
be G\^ateaux differentiable at $\signal \in  \DD$.  We define the \emph{modulus of total non-linearity} of $\func$ at  $\signal$ as $\modt_{\func}(\signal, \cdot): [0,\infty) \to [0, \infty]$,
\begin{equation}\label{eq:modulus}
	\forall t > 0 \colon \quad
	\modt_{\func}(x, t) \coloneqq \inf\set{\B_{\func} (\tilde\signal, \signal)\mid \tilde\signal \in \DD \wedge \norm{\tilde\signal - \signal} = t}
	\,.
\end{equation}
The function $\func$ is called \emph{totally non-linear} at $\signal$ if $\modt_{\func}(\signal, t) > 0$ for all $t \in (0, \infty).$
\end{definition}

The following result, due to \cite{li2018nett}, connects the convergence in absolute Bregman distance and in norm

\begin{proposition}
For $\Fo \colon D \subseteq \XX \to \R$ and and any $\signal \in D$, the followings are equivalent:
\begin{enumerate}[label=(\roman*)]
\item The function $\Fo$ is totally nonlinear at at $\signal$;
\item $\forall (\signal_n)$: ($\lim_{n \to \infty} \breg_\Fo(\signal_n,\signal) = 0 \wedge (\signal_n)$ bounded) $\Rightarrow \lim_{n\to \infty} \|\signal_n - \signal\| =0$.   
\end{enumerate}
\end{proposition}

As a consequence, we have the following convergence result in the norm topology.

\begin{theorem}[Strong Convergence] Assume that $\Ao (\signal) = \data$ has a solution, let $\reg_\theta $ be  totally nonlinear at all $\reg_\theta $-minimizing solutions of $\Ao (\signal) = \data$, and let  $(\signal_k)_{k \in \N}$, $(y_k)_{k \in \N}$, $(\alpha_k)_{k \in \N}$, $(\delta_k)_{k \in \N}$ be as in Theorem~\ref{thm:well}.
Then there is a subsequence $(\signal_{k(\ell)})_{\ell \in \N}$ of $(\signal_{k})_{k \in \N}$ and an $\reg_\theta $-minimizing solution  $\signal^\plus$ of $\Ao (\signal) = \data$ such that $\lim_{\ell \to \infty} \norm{\signal_{k(\ell)} - \signal^\plus } =0$. Moreover, if the $\reg_\theta $-minimizing solution of $\Ao (\signal) = \data$ is unique, then $\signal \to \signal^\plus$ in the norm topology.
\end{theorem}

We now focus on the convergence rate. To this end, we make the following assumptions. 
 following assumptions:
\begin{enumerate}[leftmargin=3em,label = (B\arabic*)]
\item\label{b1}  $\Y$ is a finite dimensional space;
\item\label{b2} $\reg$ is coercive and weakly sequentially lower semi-continuous;
\item\label{b4} $\reg$ is Lipschitz;
\item\label{b5} $\reg$ is G\^ateaux differentiable.
\end{enumerate}

The most restrictive condition in the above list is that $\Ao$
has  finite-dimensional range. However, this assumption holds true in practical applications such as sparse data tomography, which is a main focus of deep learning techniques for inverse problems. For infinite dimensional space result see \cite{li2018nett}.

We start our analysis with the following result.

\begin{proposition}\label{P:upperbound}
Let \ref{b1}-\ref{b5}  be satisfied and assume that
$\signal^\plus$ is an $\reg$-minimizing solution of $\Ao (\signal) = \data$. Then there exists a constant $C>0$ such that
\begin{equation*}
\forall \signal \in \X \colon \quad
\B_\reg (\signal,\signal^\plus) \leq \reg(\signal) - \reg(\signal^\plus) + C \norm{\Ao(\signal) - \Ao(\signal^\plus)} \,.
\end{equation*}
\end{proposition}
The proof follows \cite{obmann2020sparse}. We present it here for the sake of completeness. 
\begin{proof}
Let us first prove that for some constant $\gamma \in  (0, \infty)$ it holds
\begin{equation} \label{E:bounded}
\forall \signal  \in \X \colon \quad
\reg(\signal^\plus) - \reg(\signal) \leq \gamma \norm{\Ao(\signal^\plus) - \Ao(\signal) } \,.\end{equation}
Indeed, let $\Po$ be the orthogonal projection onto $\ker(\Ao)$ and define $\signal_0 = (\signal^\plus - \Po ( \signal^\plus)) + \Po (\signal)$. Then, we have $\Ao(\signal_0) = \Ao(\signal^\plus)$ and $\signal - \signal_0 \in \ker(\Ao)^\bot$.  Since the restricted operator $\Ao|_{\ker(\Ao)^\perp} \colon \ker(\Ao)^\perp \to \Y$ is injective and has finite-dimensional range, it is bounded from below by a constant $\gamma_0$. Therefore,
\begin{equation}\label{eq:aux1}
\norm{\Ao(\signal^\plus) - \Ao(\signal) } =
\norm{\Ao(\signal_0) - \Ao(\signal) }       =
\norm{\Ao(\signal_0-\signal)\| \geq \gamma_0 \|\signal_0 - \signal } \,.
\end{equation}
On the other hand, since $\signal^\plus$ is the $\reg$-minimizing solution of $\Ao (\signal) = \data$ and $\reg$ is Lipschitz, we have
$\reg(\signal^\plus) - \reg(\signal) \leq \reg(\signal_0) - \reg(x) \leq L \|\signal_0 - \signal\|$. Together with \eqref{eq:aux1}
we obtain  \eqref{E:bounded}.

Next we prove that there is a constant $\gamma_1$ such that
\begin{equation} \label{E:bounded2}
\left < \reg'(\signal^\plus), \signal^\plus - \signal \right> \leq \gamma_1 \norm{\Ao(\signal^\plus) -\Ao(\signal)} \,.
\end{equation}
Indeed, since  $\signal^\plus$ is an $\reg$-minimizing solution of $\Ao (\signal) = \data$, we obtain $\left\langle\reg'(\signal^\plus), \signal^\plus - \signal_0 \right\rangle  \leq 0$. Therefore,
\begin{multline*}\left < \reg'(\signal^\plus), \signal^\plus - \signal \right>  = \left < \reg'(\signal^\plus), \signal^\plus - \signal_0 \right> +  \left < \reg'(\signal^\plus), \signal_0 - \signal \right> \\ \leq  \left < \reg'(\signal^\plus), \signal_0 - \signal \right>  \leq \|\reg'(\signal^\plus)\| \| \signal_0 - \signal\|.\end{multline*}
Using \eqref{eq:aux1} again we obtain \eqref{E:bounded2}.

To finish the proof, we consider two cases:
\begin{itemize}
\item $ \reg(\signal^\plus) \leq \reg(\signal) \Rightarrow \abs{\reg(\signal^\plus) - \reg(\signal)} = \reg(\signal) - \reg(\signal^\plus)$
\item $\reg(\signal^\plus) \geq \reg(\signal)  \Rightarrow
\abs{ \reg(\signal^\plus) - \reg(\signal) } = \reg(\signal) - \reg(\signal^\plus)+2 (\reg(\signal^\plus) - \reg(\signal))$.
\end{itemize}
Therefore, using (\ref{E:bounded}) and (\ref{E:bounded2}), we obtain
\begin{align*} \B_\reg(\signal,\tilde \signal) &\leq \abs{\reg(\signal^\plus) - \reg(\signal)} + \abs{\left\langle \reg'(\signal^\plus), \signal - \signal^\plus\right\rangle }\\ &\leq \reg(\signal) - \reg(\signal^\plus)+ (2 \gamma + \gamma_1) \norm{\Ao (\signal) - \Ao (\signal^\plus)} \,,\end{align*}
which concludes  our proof with $C \coloneqq 2 \gamma_0 + \gamma_1$.
\end{proof}

Here is our   convergence rates result, which is an extension of~\cite[Theorem~3.1]{obmann2020sparse}. 

\begin{theorem}[Convergence rates results]Let \ref{b1}-\ref{b5}  be satisfied and suppose
  $\alpha \sim \delta$.
Then $\B_\reg(\signal_\alpha^\delta, \signal^\plus) = \mathcal{O}(\delta)$ as $\delta \to 0$.
\end{theorem}

\begin{proof}
From Proposition~\ref{P:upperbound}, we obtain
\begin{eqnarray*} \alpha \B_\Fo(\signal_\alpha^\delta,\signal^\plus) &\leq& \alpha \reg(\signal_\alpha^\delta) - \alpha \reg(\signal^\plus) + C \alpha \|\Ao(\signal_\alpha^\delta) - \Ao(\signal^\plus)\| \\
&=& \mathcal{T}_{\alpha,\delta}(\signal_\alpha^\delta) - \simm(\Ao(\signal_\alpha^\delta), \data^{\delta}) - \Big(\mathcal{T}_{\alpha,\delta}(\signal^\plus) - \simm(\Ao(\signal^\plus),\data^{\delta}) \Big) \\ &+& C \alpha \|\Ao(\signal_\alpha^\delta) - \Ao(\signal^\plus)\| \\ &\leq& \delta^2 + C \alpha \delta - \simm(\Ao(\signal_\alpha^\delta) ,\data^\delta) + C \alpha \|\Ao(\signal_\alpha^\delta)- \data^\delta\|  
\\ &\leq& \delta^2 + C \alpha \delta - \simm(\Ao(\signal_\alpha^\delta) ,\data^\delta) + C \alpha \sqrt{\simm(\Ao(\signal_\alpha^\delta),\data^\delta)} \,.
\end{eqnarray*}
Cauchy's inequality gives
$\alpha \B_\reg(\signal_\alpha^\delta,\signal^\plus) \leq \delta^2 + C \alpha \delta + {C^2 \alpha^2} / 4$.
For $\alpha \sim \delta$, we easily conclude
$\B_\reg(\signal_\alpha^\delta,\signal^\plus) = \mathcal{O}(\delta)$.
\end{proof}

\subsection{Related methods} \label{S:a}

The use of  neural networks as regularizers or similarity measures is an active research directions. Many interesting works have been done. We briefly review several techniques: variational networks \cite{kobler2017variational},  deep cascaded networks \cite{kofler2018u,schlemper2017deep} and the MODL approach \cite{aggarwal2018modl}. Further, we propose INDIE as a new operator-inversion free variant of MODL. As opposed to the discussion in Section~\ref{S:b}, these works make use of the approach (T1): employing solver-dependent training. Finally, we will  discuss a synthesis variant of the NETT framework.

\paragraph*{Variational networks:} 
Variational networks \cite{kobler2017variational}  connect variational methods and deep learning. They are based on the fields of experts model \cite{roth2005fields} and consider the Tikhonov functional     
$$   \tik_{\data,\alpha}(\signal) =  
\sum_{c=1}^{N_c} \tik_c( \signal) \coloneqq
\sum_{c=1}^{N_c} \Bigl( \sum_j \sum_i \phi_i^c((\bar K^c_i \signal)_j)+  \alpha \sum_j \sum_i \psi_i^c((K^c_i (\Ao(\signal)-\data))_j) \Bigr) \,,
$$
where  $\bar K^c_i$ and,  $K^c_i$ are learnable convolutional operators, and $\phi_i, \psi_i$ are learnable  functionals. Alternating gradient descent method for minimizing $\tik_{\data,\alpha}$ provides the update formula
\begin{equation} \label{eq:gradient} 
\signal_{n+1} = \signal_n - \eta_n \nabla_\theta \tik_{c(n)} (\signal_n)
\quad \text{ where } c(n) = 1+ (n \mod N_c) \,.
\end{equation}
Direct calculations show
$\nabla_\theta \tik_c(\signal) = \sum_i  (\bar K^c_i)^T (\phi_i^c)' (K_i^c \signal) + \Ao^T \sum_i  (K_i^c)^T (\psi_i^c)'(\bar K_i^c(\Ao (\signal)-\data))$. Minimizing the $ \tik_{\data,\alpha}$ is then replaced  by  training the neural network that consists of a $L$ blocks realizing the iterative update \eqref{eq:gradient}.

\paragraph{Network cascades:} 
Deep network cascades \cite{kofler2018u,schlemper2017deep}  alternate  between  the application of post-processing networks and so-called data consistency layers. The data consistency condition proposed in \cite{kofler2018u} for sparse data problems $\Ao  = \So  \circ \Ao_{\rm F}$  where $\So$ is a sampling operator and $\Ao_{\rm F}$ a full data forward  operator (such as the fully sampled Radon transform) takes the  form
\begin{equation} \label{eq:DC}
\signal_{n+1}= 
\Bo_{\rm F}  \left( \argmin_{\zsignal} \norm{ \zsignal - \Ao_{\rm F}( \No_{\theta(n)} (\signal_n ) )    }_2^2 + 
\alpha \norm{  \data  - \So (\zsignal)}_2^2 \right) \,,
\end{equation}
with initial reconstruction $\signal_0  =  (\Bo_{\rm F} \circ \So^* ) (\data)$, where $\Bo_{\rm F} \colon \YY \to \XX$ is a reconstruction method for the full data forward  operator and $\No_{\theta(n)}$ are networks. For example, in MRT the operator $\Bo_{\rm F}$ is the inverse Fourier transform \cite{schlemper2017deep} and in CT the operator $\Bo_{\rm F}$ can be implemented by the filtered backprojection \cite{kofler2018u}.  The resulting neural network consists of $L$ steps of \eqref{eq:DC} that can be trained end-to-end. 

\paragraph*{MODL approach:} 
The model based deep learning (MODL)  approach of \cite{aggarwal2018modl}  starts with the Tikhonov functional $\tik_{\data,\alpha} (\signal) =  \norm{\Ao(\signal) - \data}_2^2 + \alpha \norm{\signal - \No_\theta(\signal)}_2^2 $ where   $\No_\theta(\signal)$ is interpreted as  denoising network.  By designing $\No_\theta$ as a convolutional block, then $ \signal - \No_\theta(\signal)$ is a small residual network \cite{he2016deep}. The authors of \cite{aggarwal2018modl} proposed the following heuristic iterative scheme $\signal_{n+1} = 
\argmin_\signal   \norm{ \Ao (\signal) - \data }^2 + \alpha \|\signal - \No_\theta(\signal_n)\|_2^2$ based on  $\tik_{\data,\alpha}$ whose closed-form solution is 
\begin{equation}\label{E:DC} \signal_{n+1}= (\Ao^\intercal \Ao + \alpha \Id)^{-1}(\Ao^\intercal \data + \alpha \, \No_\theta(\signal_n)) \,.
\end{equation} 
Concatenating these steps together one arrives at a deep neural network.
 Similar  to network cascades, each block  \eqref{E:DC} consists of a trainable layer $\zsignal_n = \Ao^\intercal \data + \alpha \No_\theta(\signal_n)$ and a non-trainable data consistency layer $ \signal_{n+1} = (\Ao^\intercal \Ao + \lambda \Id)^{-1}(\zsignal_n)$.

\paragraph*{INDIE approach:}
 
Let us present an alternative to the above procedures, inspired by \cite{daubechies2004iterative}. Namely we  propose the  
iterative update 
\begin{align*}
&\signal_{n+1} = \argmin \tikl_n(\signal) \\
& \tikl_n(\signal)  \coloneqq \norm{\Ao(\signal) -\data}^2 + \alpha \|\signal - \No_\theta(\signal_n)\|^2 + C \norm{ \signal-\signal_n }^2 -  \norm{\A(\signal-\signal_n)}^2 \,.
 \end{align*}
 Here  the constant $C > 0$ is an upper bound for the operator norm $\norm{\Ao}$. Elementary  manipulations show the identity  
\begin{multline*} \tikl_n(\signal) =-2\Bigl\langle \Ao^\intercal(\data- \Ao(\signal_n))+\alpha \No_\theta (\signal_n) + C \signal_n, \signal \Bigr\rangle  \\  + (\alpha+C) \|\signal\|^2 - (\alpha \|\No_\theta(\signal_n)\| + C \|\signal_n\|^2) - \|\A(\signal_n)\|^2 + \norm{\data}^2 \,.
\end{multline*}
The minimizer of $\tikl_n$ can therefore be computed  explicitly by setting the gradient of the  latter expression to zero.  This results in  the proposed network block 
\begin{equation} \label{eq:linh}
\signal_{n+1} = \frac{1}{\alpha+C}  \, \Bigl( \Ao^\intercal(\data- \Ao(\signal_n))+\alpha \No_\theta (\signal_n) + C\signal_n \Bigr) \,.
\end{equation}
This results at a deep neural network similar to the MODL iteration. However, each block in  \eqref{eq:linh} is clearly simpler than the blocks in  \eqref{E:DC}. In fact, opposed to MODL our proposed learned iterative scheme does not require costly matrix inversion. We name the resulting iteration  INDIE (for \textbf{i}\textbf{n}version-free \textbf{d}eep \textbf{i}t\textbf{e}rative) cascades. We consider the numerical comparison  of MODL and INDIE as well as the theoretical analysis of both architectures interesting lines of future research.

\paragraph*{Learned synthesis regularization:} 

Let us finish this section by pointing out that regularization by neural network is not restricted to the form (\ref{eq:nett}). For example, one can consider the synthesis version, which reads \cite{obmann2020deep}
\begin{equation} \label{eq:desryre} 
\signal^{\rm syn}  = 
\Do_\theta \Bigl( \argmin_\xi  \norm{ \Ao \circ \Do_\theta(\xi) - \data}^2 + \alpha \sum_{\lambda \in \Lambda} \omega_\lambda |\xi_\lambda|^p \Bigr) \,,\end{equation}
where $\Lambda$ is a countable set, $1 \leq p <2$,  and $\D_\theta \colon \ell^2(\Lambda) \to \XX$ is  a learned operator that performs nonlinear synthesis of $\signal$. Rigorous analysis of the above formulation was derived in \cite{obmann2020deep}. 

Finally,  note that one can generalize the frameworks \eqref{eq:nett} and \eqref{eq:desryre} by allowing the involved neural networks to depend on the regularization parameter $\alpha$ or the noise level $\delta$. The dependence on $\alpha$ has been studied in, for example, \cite{obmann2020deep}. The dependence on $\delta$ can be realized by mimicking the Morozov's stopping criteria when training the neural networks, either independently or together with the optimization problem. In the later case, $\delta$ can help to decide the depth of the unrolled neural network.

\section{Conclusion and outlook}
\label{sec:conclusion}

Inverse problems are central to the solving a wide range of important practical problems within and outside of imaging and computer vision. Inverse problems are characterized by the ambiguity and instability of their solution. Therefore, stabilizing solution methods based on regularization techniques are necessary to  solve them in reasonable way. In recent years, neural networks and deep learning have emerged as the rising stars  for  the solution of inverse problems. In this chapter, we have developed the mathematical foundations  for solving inverse problems with deep learning. In addition, we have shown stability and convergence for selected neural networks to solve inverse problems. The investigated methods, which combine the strengths of both worlds, are regularizing null-space networks and the NETT (Network-Tikhonov) approach for inverse problems.

\end{document}